\documentclass[amsthm,10pt,onecolumn]{elsart3p}

\usepackage{mathrsfs}
\usepackage[all]{xy}
\usepackage{graphics}
\usepackage{amsmath}
\usepackage{amsfonts}
\usepackage{amssymb}
\usepackage{graphicx}
\usepackage{enumitem}
\usepackage{natbib}

\def\dto{\dashrightarrow}
\def\hto{\hookrightarrow}
\def\lto{\longrightarrow}

\def\X{\textbf{X}}

\def\t{\textbf{t}}

\def\im{\textnormal{im}}
\def\min{\textnormal{min}}

\def\res{\textnormal{Res}}

\def\codim{\textnormal{codim}}
\def\coker{\textnormal{coker}}
\newcommand\Spec{\textnormal{Spec}}
\newcommand\Proj{\textnormal{Proj}}
\newcommand\BiProj{\textnormal{BiProj}}
\newcommand\MProj{\textnormal{MultiProj}}
\newcommand\length{\textnormal{length}}
\def\div{\textnormal{div}}
\newcommand\gen[1]{\left\langle#1\right\rangle}

\def\PP{\mathbb P}

\def\ZZ{\mathbb Z}
\def\NN{\mathbb N}

\def\mmm{\mathfrak m}

\def\k.{\mathcal{K}_{\bullet}}
\def\M.{\mathscr{M}_\bullet}
\def\Z.{\mathscr{Z}_\bullet}
\def\B.{\mathscr{B}_\bullet}

\def\kf{\k.(\textbf{f};A)}

\def\ZZZ{\mathscr Z}

\newcommand\f{\textbf{f}}
\newcommand\bt{\textbf{t}}

\newcommand\RRes{\res_A(L_0,\dots,L_n)}

\newcommand\CC{\mathbb C}
\newcommand\Aa{\mathfrak A}

\def\AA{\mathbb A}
\newcommand\pp{\mathfrak p}
\newcommand\qq{\mathfrak q}
\newcommand\kk{\kappa}

\newcommand\xx{x_0,y_0}

\newcommand\zz{x_2,y_2}

\newcommand\ffa{f_0,g_0}
\newcommand\ffi{f_i,g_i}
\newcommand\ffn{f_n,g_n}
\newcommand\xxi{x_i,y_i}
\newcommand\xxn{x_n,y_n}

\newcommand\fxi{f_iy_i-g_ix_i}

\def\Za{\Z.(f_0,g_0)}

\def\Zi{\Z.(f_i,g_i)}

\newcommand\Symffii{\textnormal{Sym} _{(f_i,g_i)} (A[\xxi])}

\def\P1{\PP^1}
\def\p2{\PP^2}

\def\Sym{\textnormal{Sym}}

\def\OO{\mathcal O}

\def\UU{\mathcal U}

\def\EEE{\mathscr E}

\def\ZZZ{\mathscr Z}

\def\BB{\mathcal B}
\def\HH{\mathcal H}

\begin{document}

\begin{frontmatter}

\title{The implicitization problem for $\phi: \PP^n \dashrightarrow (\PP^1)^{n+1}$}

\author{Nicol\'{a}s Botbol}
\address{Departamento de Matem\'atica\\
FCEN, Universidad de Buenos Aires, Argentina \\
\& Institut de Math\'ematiques de Jussieu \\
Universit\'e de P. et M. Curie, Paris VI, France \\
E-mail address: nbotbol@dm.uba.ar
}

\thanks{The author was partially supported by: project ECOS-Sud A06E04, CONICET, UBACYT X042, PAV 120 - 03 (SECyT), and  ANPCyT PICT 20569, Argentina.}


\begin{abstract}
We develop in this paper methods for studying the implicitization problem for a rational map $\phi: \PP^n \dashrightarrow (\P1)^{n+1}$ defining a hypersurface in $(\P1)^{n+1}$, based on computing the determinant of a graded strand of a Koszul complex. We show that the classical study of Macaulay Resultants and Koszul complexes coincides, in this case, with the approach of approximation complexes and we study and give a geometric interpretation for the acyclicity conditions. 

Under suitable hypotheses, these techniques enable us to obtain the implicit equation, up to a power, and up to some extra factor. We give algebraic and geometric conditions for determining when the computed equation defines the scheme theoretic image of $\phi$, and, what are the extra varieties that appear. We also give some applications to the problem of computing sparse discriminants.
\end{abstract}


\end{frontmatter}

\medskip

\section{Introduction}

In this work we study the implicitization problem for a finite rational map $\phi: \PP^n \dashrightarrow (\P1)^{n+1}$ over a field $k$, hence, its image is a hypersurface in $(\P1)^{n+1}$. Having a rational map $\phi: \PP^n \dashrightarrow (\P1)^{n+1}$ is equivalent to having $n+1$ pairs of homogeneous polynomials $f_i, g_i $ of the same degree $d_i$, for $i=0,\hdots,n$, $f_i,g_i $ with no common factors.

We show that the classical study of Macaulay Resultants and Koszul complexes coincides with the new approach introduced by L. Bus\'e and J.-P. Jouanolou in \cite{BJ} and developed by them and M. Chardin in \cite{Cha3,BC,Buse1,BCJ}, by means of approximation complexes, defined by J. Herzog, A. Simis and W. Vasconcelos in \cite{HSV1} and \cite{HSV2}.

The process consists of computing the implicit equation by means of the classical methods of elimination theory adapted for this case. More precisely, we will consider the multigraded $k$-algebra $\BB$ that corresponds to the incidence variety associated to the given rational map $\phi$. This algebra can be presented as a quotient of the polynomial ring $R$ in all the groups of variables, by some linear equations $L_0,\hdots,L_n$. Consequently we propose as a resolution for $\BB$, the Koszul complex $\k.^R(L_0,\hdots,L_n)$, denoted by $\k.$, and we study and give a geometric interpretation of its acyclicity conditions. 

In this case, we obtain the implicit equation (up to a power) by taking the determinant of a suitable strand of a multigraded resolution, that is:
\[
 H^{\deg(\phi)}=\textnormal{Res}(L_0,\hdots,L_n)=\det((\k.)_\nu),\qquad\textnormal{for }\nu\gg 0.
\]

Later, we analyze the geometrical meaning of these results. We give algebraic and geometric conditions for knowing when, the computed equation defines the scheme theoretic image of $\phi$. And, when it is not, we present a careful analysis of the extra varieties that appear.

Finally, we give some applications to the problem of computing sparse discriminants, or $A$-discriminants (cf. \cite{CD}), by means of implicitization techniques, that were one of the reasons for developing this technique. 

The key point of our approach is that the hypotheses in our main Theorem \ref{teoRes} are generically satisfied for rational parametrizations whose coordinates are rational functions of degree zero (defining naturally a rational morphism to $(\P1)^{n+1}$), while if we try to reduce this situation to the standard case of a rational morphism to $\PP^n$ by taking a common denominator, bad base points appear in general and the results developed in \cite{BJ} do not apply.


\section{Preliminaries on commutative algebra}

\subsection{The Koszul complex}

We present here some basic tools of commutative algebra we will need for our purpose, starting by a classical result due to Hurwitz (1913). He showed that in the generic case, when there are at least as many variables as homogeneous polynomials, the Koszul complex is acyclic.

Let $R$ be a ring, and $P_0,\hdots, P_n$ a sequence in $R$ generating an ideal that we will denote by $I$. Denote by $\k.$ the associated Koszul complex $\k.(P_0,\hdots, P_n;R)$: 
\[
\k.: 0\to \bigwedge^{n+1}R^{n+1}\stackrel{\partial_{n+1}}{\to} \cdots\to \bigwedge^{i+1}R^{n+1} \stackrel{\partial_{i+1}}{\to} \bigwedge^{i}R^{n+1}\to \cdots\stackrel{\partial_1}{\to} \bigwedge^{0}R^{n+1}\to 0, 
\]
where the morphisms $\partial_{i+1}:\bigwedge^{i+1}R^{n+1}\to \bigwedge^{i}R^{n+1}$,  are defined in such way that the element $e_{k_{1}}\wedge\cdots\wedge e_{k_{i+1}}\in\bigwedge^{i+1}R^{n+1}$ is mapped to $\sum_{j=1}^{i+1}(-1)^{j-1}P_{k_{j}}e_{k_{1}} \wedge\cdots\wedge\widehat{e_{k_{j}}}\wedge\cdots\wedge e_{k_{i+1}}$.

If $R$ is graded and every $P_i$ is homogeneous of degree $d_i > 0$, this complex inherits the grading. If we introduce in $\k.$ the grading given by $\deg(e_{k_{1}}\wedge\cdots\wedge e_{k_{i+1}})=d_{k_{1}}+\cdots+ d_{k_{i+1}}$, the differentials are of degree $0$. It is a well known fact that, $P_0,\hdots,P_n$ is a regular sequence of homogeneous polynomials of positive degree if and only if $\k.$ is acyclic (see for instance \cite[Sec. 9, N.7, Cor. 2]{Bou10}). Naturally, if $R$ is graded or multigraded, and the $P_i$ are homogeneous with respect to the grading or multigrading, this complex inherits this multigrading, and the same statement of acyclicity still holds. We will apply this to the particular case where $R=k[t_0,\hdots,t_n]\otimes_k k[X_0,Y_0,\hdots,X_n,Y_n]$, and $k$ any field. Here, the ring is naturally bigraded, and can also be seen as $\NN\times\NN^{n+1}$-graded, by considering one grading given by $t_0,\hdots,t_n$, and the $\NN^{n+1}$-grading given by the $n+1$ pairs $X_i,Y_i$. The polynomials $P_i$ will be hence multihomogeneous of multidegree $(d_i,0,\hdots,1,\hdots,0)$, precisely $P_i\in k[t_0,\hdots,t_n]_{d_i}\otimes_k k[X_i,Y_i]_1$, and will be called $L_i$ because of their linearity in the second group of variables.

Assume $A$ is a noetherian commutative ring, and $R=A[X_1,\hdots,X_m]$. Set $I=(P_0,\hdots,P_n)$, with $P_i=\sum_{j=1}^m m_{ij}X_j$, $m_{ij}\in A$. A theorem due to L. Avramov \cite{avr} gives necessary and sufficient conditions for $I$ to be a complete intersection in $R$ in terms of the depth of certain ideals of minors of $M:=(m_{ij})_{i,j}\in Mat_{m,n+1}(A)$.

\medskip

\begin{thm}\label{avramov}[L. Avramov] The ideal $I$ is a complete intersection in $R$ if and only if for all $r\in \{0,\hdots,n\}$, $\codim_A(I_r)\geq (n+1)-r+1$. Here $I_r=I_r(M)$ is the ideal of $A$ generated by the $r\times r$ minors of $M$, for $0\leq r \leq r_0:=\min\{n+1,m\}$. We define $I_0:=A$ and $I_r=0$ for $r>r_0$.
\end{thm}

\medskip

\noindent For a proof we refer the reader to \cite[Prop. 1]{avr}.

\medskip

Later in this article, we will apply this result when $A = k[t_0,...,t_n]$, $R = A[X_0,Y_0,...,X_n,Y_n]$, 
and $M$ is the $2(n+1)\times (n+1)$ matrix

\begin{equation}\label{matrizota}
 M=\left(\begin{array}{ccc}-g_0&\hdots&0\\ f_0&\hdots&0\\
\vdots&\ddots&\vdots\\ 0&\hdots&-g_n\\ 0&\hdots&f_n  \end{array}\right),
\end{equation}
that defines a map of $A$-modules $\psi: A^{n+1} \to \bigoplus_{i=0}^n A[x_i,y_i]_1\cong A^{2(n+1)}$, and we see that the symmetric algebra $\Sym_A(\coker(\psi))\cong A[\X]/(P_0,\hdots,P_n)$, where $\X$ stands for the variables $x_0,y_0,\hdots,x_n,y_n$. As $A$ is a graded $k$-algebra, $\Sym_A(\coker(\psi))$ is naturally multigraded and the graph of $\phi$ is an irreducible component of $\MProj(\Sym_A(\coker(\psi)))\subset \PP^n\times(\P1)^{n+1}$.

\subsection{Approximation complexes}

We give a brief outline of the construction of the approximation complexes of cycles, and  show that the complex $\Z.$ coincides in this particular case (under weak hypotheses) with a certain Koszul complex.

Assume we are given a sequence $\f:=(f_0,\hdots,f_{n+1})$ of homogeneous elements of degree $d$ over the graded ring $A=k[t_0,\hdots,t_n]$, generating an ideal $I$. Consider the two Koszul complexes $\k.^A(\f)$ and $\k.^R(\f)=\k.^A(\f)\otimes_A R$. We have also, in $A[x_0,\hdots,x_{n+1}]$, another relevant sequence to consider, let's call it $\X:=(x_0,\hdots,x_{n+1})$, and we can consider also the corresponding Koszul complex $\k.^R(\X)$, which is of course acyclic because of the regularity of the sequence $\X$.

A straightforward computation permits to verify that their differential anticommutes, i.e. $d_f \circ d_X + d_X \circ d_f=0$, and this implies that $d_X$ induces a differential on the cycles, boundaries and homologies of the complex $\k.^R(\f)$ . We obtain in this way three complexes noted by $\Z.$, $\B.$ and $\M.$, called respectively the approximation complexes of cycles, boundaries and homologies.

\medskip

\begin{rem}
Recall that the homology modules of these complexes are, up to isomorphism, independent of the choice of generators for $I$.
(See for instance \cite[Cor 3.2.7]{Vas1}).
\end{rem}

\medskip

A more explicit description of the $\ZZZ$-complex is the following:
\[
 \Z.(\f): \ 0\to Z_{n+1}(\k.^R(\f))[d(n+1)](-n-1) \stackrel{d_X}{\to} \hdots
\stackrel{d_X}{\to} Z_1(\k.^R(\f))[d](-1) \stackrel{d_X}{\to} A[\X] \to 0
\]
where $Z_{i}(\k.^R(\f))$ stands for the $i$th cycle of the complex $\k.^R(\f)$, and $[-]$ and $(-)$ are the shifts in the two groups of variables $t_0,\hdots,t_n$ and $x_0,\hdots,x_{n+1}$.

This complex has as objects the bigraded modules 
\begin{equation}\label{Z con graduacion}
\ZZZ_i=Z_i(\kf)[di]\otimes_A R(-i),
\end{equation} 
and in the future let $Z_i$ denote the module $Z_i(\kf)$.

\medskip

In the case of a two-generated ideal, one has the following:

\medskip

\begin{prop}\label{iso long 1}
 With the notation above, if the sequence $\{f,g\}$ is regular, then there exists a bigraded isomorphism of complexes
\[
 \Z.(f,g) \cong \k.(L;A[x,y]),
\]
where $L:=yf-xg$.
\end{prop}

\begin{proof}
Given the sequence $\{f, g\}$ the approximation complex is:
\[
\Za: \ 0\to Z_1[d]\otimes_A A[x, y](-1) \stackrel{(x, y)}{\longrightarrow} Z_0\otimes_A A[x, y] \to 0. 
\]

As the sequence $\{f,g\}$ is regular, $H_1(\k.^A(\f))=0$, hence $Z_1=(-g,f)A\cong A$ by the isomorphism $\psi$ that maps $a\in A$ to $(-g.a,f.a)\in Z_1$, given by the left morphism of the acyclic Koszul complex $\k.^A(\f)$. Tensoring with $A[x, y]$ we get an isomorphism of $A$-modules, $\ZZZ_1\cong A[x, y]$. 

The commutativity of the following diagram shows that $\Z.(f,g) \cong \k.(L;A[x,y])$
\[
\xymatrix@1{ 
\Z. :\  0\ar[r] & Z_1[d]\otimes_A A[x, y](-1)\ar[r]^(.55){(x,y)} &Z_0[d]\otimes_A A[x, y] \ar[r]& 0\ \\
\k. :\  0\ar[r] & A[x,y][-d](-1) \ar[r]^(.6){L}\ar[u]^{\psi\otimes_A 1_{A[x,y]}} &A[x,y] \ar[r]\ar[u]^{=}& 0,} 
\]
where $\k.$ denotes $\k.(L;A[x,y])$, $[-]$ denotes the degree shift for the grading on $A$ and $(-)$ the shift in the variables $x,y$.
\end{proof}

\subsection{Elimination theory, the Macaulay Resultant and the U-Resultant}\label{sec-elim}

We recall here some basic properties of elimination theory, and also basic facts about resultants that were introduced by F. S. Macaulay in 1902, an later formalized by J.-P. Jouanolou in ``Le formalisme du R\'esultant", cf. \cite{Jou3}. This resultant corresponds to a generalization of the Sylvester Resultant of two homogeneous polynomials in two variables. Here, we present a brief outline in elimination theory and its classical results.

Let $\ZZ$ be the ring of integers, $\bt=t_0,\hdots, t_n$ $n+1$ indeterminates. Let $d_j$ be $n+1$ non-negative integers and let $P_j=\sum_{|\alpha|=d_j}X_{j,\alpha}\bt^\alpha$ be $n+1$ homogeneous polynomials of degree $d_j$ in the variables $t_0,\hdots,t_n$ generating an ideal $I$.

Let us write $B:=\ZZ[X_{j,\alpha}\ |\ j+0,\hdots,n, |\alpha|=d_j]$, and $R:= B[t_0,\hdots,t_n]$. Also let us call $S$ the affine spectrum of $B$, that is $S:= \Spec(B)$. With the assumption $\deg(t_i)=1$, we have that $R/I$ is a $\ZZ$-graded $B$-algebra. So we can consider the projective $S$-scheme $Z:=\Proj(R/I)\hto \PP^n_S$, the incidence variety:
\[
 Z=\{(\X,\t)\in \PP^n\times S\ |\ P_j(\X,\t)=0,\ \forall j\}\stackrel{i}{\hto} \PP^n_S.
\]
We denote by $\Aa$ the kernel of the following canonical map of rings $\Aa:=\ker(B=\Gamma(S,\OO_S)\to \Gamma(Z,\OO_Z))\cong (H_\mmm^0(R/I))_0$, and set $T:=\Spec(B/\Aa)\stackrel{j}{\hto}\Spec(B)=S$. We have the following commutative diagram of schemes:
\begin{equation}\label{diag-elim}
 \xymatrix@1{ Z=\Proj(R/I)\ \ar@{^{(}->}[r]^{i}\ar[d]^{p} &\PP^n_S=\PP^n \times_\ZZ S \ar[d]^{p}   \\T=\Spec(B/\Aa) \ \ar@{^{(}->}[r]^{j}&S=\Spec(B)  }
\end{equation}

\begin{rem}\label{rem-elim}
 The scheme $S$ parametrizes sequences of polynomials $(P_0,\hdots,P_n)\subset \ZZ[t_0,\hdots,t_n]$. A closed point $x=V(\mmm)$ of $S$ belongs to $T$ if and only if the sequence $(P_0,\hdots,P_n)$ of associated polynomials has a common root in $\PP^n_k$ for some extension $k$ of $S/\mmm$.
\end{rem}

\medskip

\begin{thm}\label{teo-elim}[J.-P. Jouanolou] With the notation above, the following statements are satisfied:
\begin{enumerate}
 \item $\Aa$ is a principal ideal in $B$, whose generator will be denoted by $\res_\ZZ(P_0,\hdots,P_n)$.
 \item The element $\res(P_0,\hdots,P_n)$ of $B$ is not a zero divisor, and for all $j=0,\hdots,n$ is homogeneous of degree $\prod_{i\neq j}d_i$ with respect to the variable $X_j$. 
 \item $\Spec(B/(P_0,\hdots,P_n))$ is geometrically irreducible in $S$. Moreover, for any morphism of commutative rings, $\ZZ[X_{j,\alpha}| \ |\alpha|=d_j] \to k$,
we define $\res_k(\epsilon(P_0),\hdots,\epsilon(P_n)):=\epsilon(\res_\ZZ(P_0,\hdots,P_n))$ where $\epsilon$ is extended to a morphism for $B$ to $k[t_0,\hdots,t_n]$, linearly in the $t_i$'s.
\item if $k$ is a field, $\res_k(\epsilon(P_0),\hdots,\epsilon(P_n))=0$ if and only if $\epsilon(P_0),\hdots,\epsilon(P_n)$ have a common root (different from zero) in an extension of $k$.
\end{enumerate}
\end{thm}
\medskip
The original presentation of this result in these terms is in \cite[Prop 2.3]{Jou3}.
\medskip

\begin{rem}\label{remResolResul}
 If the sequence $\{P_0,\hdots,P_n\}$ is regular, $\res(P_0,\hdots,P_n)$ can be computed as the determinant of $\k.(P_0,\hdots,P_n;R)_\nu$, for $\nu>\eta:=\sum(d_i-1)$ (cf. \cite{Dema} or \cite{Cha1}).
\end{rem}


\section{The algebraic approach}

We will establish here the relation between approximation complexes, tensor products of them, and some Koszul complex we will present below. 

\medskip

Assume that the sequence $\{f_i,g_i\}$ is regular for all $i\in \{0,\hdots, n\}$. Then, as all $n+1$ Koszul complexes are acyclic, we have isomorphisms between $A$ and the respective modules of cycles, as the mentioned in the previous section, in Proposition \ref{iso long 1}. 

\medskip
\begin{defn}\label{defTotAC}
 Write $\k.$ the Koszul complex associated to the $n+1$ polynomials $L_i:=\fxi$. Denote by $\Z.$ the complex obtained by tensoring the $n+1$ approximation complexes $\Zi$ over $A$, namely $\Z.:=\bigotimes_{i=0}^{n}\Zi$.
\end{defn}

\medskip

\begin{prop}\label{iso p1x...xp1}
 With the notation above, if the sequences $\{f_i,g_i\}$ are regular for all $i\in \{0,\hdots, n\}$, then there exists an isomorphism of $A$-complexes $\Z.\cong \k.$.
\end{prop}

\begin{proof}
It is enough to see that by definition of $\Z.$ and $\k.$, and Proposition \ref{iso long 1}, we have 
\[
\Z.:=\bigotimes_{i=0}^{n}\Zi\cong \bigotimes_{i=0}^{n}\k.(L_i;A[\xxi])\cong \k. \qedhere
\]\end{proof}

\noindent From Proposition \ref{iso p1x...xp1} we deduce 

\medskip

\begin{cor}
We can resolve the algebra 
\begin{equation}\label{defBiso}
 \BB:=\bigotimes_A \Symffii\cong \frac{A[\X]}{(\fxi)_{i=0,\hdots,n}},
\end{equation}
by means of the $\Z.$ complex, or, equivalently when the hypothesis of Proposition \ref{iso p1x...xp1} are satisfied, by means of the Koszul complex $\k.$.
\end{cor}

\medskip

We have the following result.

\medskip

\begin{lem}
 The $\Z.$ complex satisfies 
\begin{enumerate}
 \item $H_0(\Z.)\cong H_0(\k.)\cong \frac{A[\X]}{(\fxi)_{i=0,\hdots,n}}$, and 
 \item $\Z.$ is acyclic if and only if the Koszul complex $\k.(L_0,\hdots, L_n;A[\X])$ is, that is, if and only if the sequence $L_0,\hdots, L_n$ is regular.
\end{enumerate}
\end{lem}

\medskip

\begin{rem}
We said that Avramov's criterion, stated in Theorem \ref{avramov}, gives a necessary and sufficient condition for the acyclicity of $\k.$ for being acyclic. That is, conditions on $\BB$ for being a complete intersection.
\end{rem}

As mentioned in Section \ref{sec-elim}, the resultant of $L_0,\hdots,L_n$ as polynomials in the variables $t_0,\hdots,t_n$, can be computed as a MacRae invariant of $(R/I)_\nu$, for $\nu>\sum (d_i-1)$, which is the determinant of a suitable resolution of $R/I$ in degree $\nu$. Observe that these two complexes are naturally bigraded, one grading corresponds to the $\X$ variables, and the other to the $\t$ variables. The acyclicity condition on $\Z.$ is applied to the first group, and the notation $(R/I)_\nu$ stands for the grading on the second group. Hence, for a fixed $\nu$, $(\Z.)_\nu$ is a resolution of the $k[\X]$-module of $(R/I)_\nu$.

Consequently we get by this method a multi-homogeneous generator of the ideal $\Aa\subset k[\X]$, that is, we have the following implicitization result:

\medskip

\begin{cor} If the sequence $L_0,\hdots, L_n$ is regular, then the complex 
\[
 (\k.(L_0,\hdots,L_n;A[\X]))_\nu : 0\to (\ZZZ_{n+1})_\nu \stackrel{\partial_{n+1}}{\longrightarrow}\cdots \stackrel{\partial_2}{\longrightarrow} (\ZZZ_1))_\nu\stackrel{\partial_1}{\longrightarrow} (\ZZZ_0)_\nu \to 0
\]
is acyclic for all $\nu> \eta:=\sum (d_i-1)$.

Moreover, with the notation of Section \ref{sec-elim}, $\det((\k.(L_0,\hdots,L_n;A[\X]))_\nu)$ gives an implicit equation for the closed image of $p: Z=\BiProj(R/I) \to \tilde{S}=\Proj(B)$.
\end{cor}

\begin{proof}
 The result follows from Remark \ref{remResolResul} and Theorem \ref{teo-elim}. For instance, by the universality of the resultant (cf.\ Theorem \ref{teo-elim} parts (iii) and (iv)) it is enough to see this for generic coefficients. By Remark \ref{remResolResul} $\det((\k.(L_0,\hdots,L_n;A[\X]))_\nu)$ computes $\res_A(L_0,\hdots,L_n)$ for $\nu>\sum (d_i-1)$. Again, by Theorem \ref{teo-elim}, this resultant computes the divisor obtained as the projection of the incidense scheme according to Diagram \eqref{diag-elim} in the projective context:
\[
  \xymatrix@1{ Z=\BiProj(R/I)\ \ar@{^{(}->}[r]^{i}\ar[d]^{p} &\PP^n_{\tilde{S}}=\PP^n \times_\ZZ \tilde{S} \ar[d]^{p}   \\T=\Proj(B/\Aa) \ \ar@{^{(}->}[r]^{j}&\tilde{S}=\Proj(B).} 
\]
The results follows from the fact that the image of this projection is the closed image of $p: Z=\BiProj(R/I) \to \tilde{S}=\Proj(B)$.
\end{proof}

In the next section we will focus on the geometric interpretation of this fact, and in reinterpreting this result in terms of the scheme theoretic image of $\phi$. 
\medskip


\section{The geometric approach}

\subsection{The implicit equation as a resultant}

In this section we will focus on the geometrical aspects related to the acyclicity of the Koszul complex $\k.$, and the nature of the base locus of the rational map $\phi:\PP^n\dto (\P1)^{n+1}$. Recall $\phi$ is defined by $(t_0:\hdots:t_n)\mapsto (f_0:g_0)\times\hdots \times(f_n:g_n)$ where $f_i, g_i$ are homogeneous polynomials of degree $d_i$. Write $I^{(i)}:=(f_i, g_i)$ the homogeneous ideal of $A$ defined by $f_i, g_i$.

From Proposition \ref{iso p1x...xp1}, the regularity of the sequence $\{\ffi\}$ for all $i$ implies that the complex $\k.$ coincides with the tensor product of the approximation complex of cycles, $\Z.$ (cf. Definition \ref{defTotAC}). Hence, we will focus on the conditions on $\k.$ for being acyclic, and on its geometrical interpretation.

As we mentioned above, we want here to use a suitable complex for computing $\Aa:=\res_A(L_0,\hdots,L_n)$. Also we will explain that under certain conditions on the codimension of some ideals of minors of the matrix $M$ defined in \eqref{matrizota}, we can assure that this resultant is exactly an implicit equation of the scheme theoretic image of $\phi$ to a certain power, $\deg(\phi)$.

\medskip

Write $X$ for the closed subscheme of $\PP^n$, given by the common zeroes of all $2(n+1)$ polynomials, and denote by $W$ the base locus. That is: 
\begin{equation}\label{WX}
X:=\Proj(A/\sum_{i}I^{(i)}), \textnormal{ and } \qquad  W:=\Proj(A/\prod_{i}I^{(i)})
\end{equation}
Set-theoretically we have that $X=\bigcap_{i=0}^n V(f_i,g_i)$ and $W=\bigcup_{i=0}^n V(f_i,g_i)$ inside $\PP^n$. Clearly, we always have $X\subset W$.

\medskip

Recall that given a matrix $M$ as described in (\ref{matrizota}), we write $I_r=I_r(M)$ for the ideal of $A$ generated by the $r\times r$ minors of $M$, for $0\leq r \leq r_0:=\min\{n+1,m\}$, where $I_0:=A$ and $I_r=0$ for $r>r_0$. From Theorem \ref{avramov} we have a condition for the acyclicity of the Koszul complex $\k.$ in terms of the ideals $I_r$.
\medskip

For instance, as $A$ is a Cohen Macaulay ring, codimension coincides with depth. In particular, when $n=2$, we have, from \ref{avramov}, that $\codim_{\p2}(V(I_1))\geq 3$, this is $X:=V(I_1)=V(I^{(0)}+I^{(1)}+I^{(2)})=\emptyset$, and that $\codim_{\p2}(V(I_2))\geq 2$, which implies that $V(I_2)=\{p_1,\hdots,p_s\}$ is a finite set. As $V(I_3)=W$ is the base locus, $\codim_{\p2}(V(I_3))\geq 1$ is satisfied when the base points of $\phi$ have codimension less than or equal to $1$. More generally:

\medskip

\begin{thm}\label{teoRes} Let $\phi: \PP^n\dashrightarrow (\P1)^{n+1}$ be, as in the previous section, a finite rational map given by $(t_0:\hdots:t_n)\mapsto (\ffa)\times\hdots \times(\ffn)$ where $\ffi$ are homogeneous polynomials of degree $d_i$, and for all $i=0,\hdots,n$ not both $\ffi$ are the zero polynomial. Let us denote by $A$ the polynomial ring $k[t_0,\hdots, t_n]$, with $L_i$ the expression $\fxi$, for $i=0,\dots,n$, and with $I^{(i)}$ the ideal $(\ffi)\subset A$.

\begin{enumerate}
\item The following statements are equivalent:

     \begin{enumerate}
     \item \label{teo-kos} the Koszul complex $\k.$ is a free resolution of $\BB:=\frac{A[\xx,\hdots,\xxn]}{(L_0,\hdots,L_n)}$.
     \item \label{teo-dep} $\codim_A(I_r)\geq (n+1)-r+1$ for all $r=1,\hdots,n+1$.
     \item \label{teo-cod} all the following statements are true:
	\begin{enumerate}
        \item $\bigcap_{i=1}^n V(I^{(i)})=\emptyset$.
        \item $\#(\bigcap_{i < j}( V(I^{(i)}\cdot I^{(j)})))<\infty$.
        \item $\dim(\bigcap_{i<j<k}( V(I^{(i)}\cdot I^{(j)}\cdot I^{(k)})))\leq 1$.

	$\vdots$

        \item[(n)] $\dim(\bigcap_{i}( V( I^{(0)}\cdots \hat{I^{(i)}}\cdots I^{(n)})))\leq n-2$.
       \end{enumerate}

     \end{enumerate}
\item If any (all) of the items before are satisfied, then:

     \begin{enumerate}
	 \item \label{teo-res} The (multi)homogeneous resultant $\res_{A,d_0,\hdots,d_n}(L_0,\dots,L_n)$ is not the zero polynomial in $A$.
	 \item \label{teo-equ}Denote by $H$ the irreducible implicit equation of the scheme theoretic image $\HH$ of $\phi$. If for all $\{i_0,\hdots, i_k\}\subset \{0,\hdots,n\}$ we have that $\codim_A(I^{(i_0)}+\cdots +I^{(i_k)})> k+1$, then, for $\nu>\eta=\sum_i (d_i-1)$, 
\[
 \det((\k.)_\nu)=\res_A(L_0,\hdots,L_n)=H^{\deg(\phi)}.
\]
     \end{enumerate}
\end{enumerate}
\end{thm}

\begin{proof}
 
(i)\eqref{teo-kos} $\Leftrightarrow$ (i)\eqref{teo-dep} follows from Avramov's Theorem \ref{avramov}.

(i)\eqref{teo-dep} $\Leftrightarrow$ (i)\eqref{teo-cod} Note that each $r\times r$-minor of $M$ can be expressed as a product of $r$ polynomials, where for each column we choose either $f$ or $g$. Then, the ideal of minors involving the columns $i_0,\hdots,i_{r-1}$ coincides with the ideal $I^{(i_0)}\cdots I^{(i_{r-1})}$. As we assumed that for any $i$ $f_i\neq 0$ and $g_i\neq 0$, the condition $\dim( V(I^{(0)}\cdots I^{(n)}))\leq n-1$ is automatically satisfied.

(i)\eqref{teo-kos} $\Rightarrow$ (ii)\eqref{teo-res} this is a classical result first studied by J.-P. Jouanolou in \cite[$\S$3.5]{Jou2}, reviewed in \cite{GKZ}, and also used by L. Bus\'e, M. Chardin and J-P. Jouanolou, in their previous work in the area.

(ii)\eqref{teo-equ} Let us denote by 
\[
Z=\{(t,x): t=(t_0,\hdots,t_n), x=(\xx,\hdots,\xxn), L_i(t,x)=0, \ \forall i=0,\hdots,n\}. 
\]
As all the polynomials $L_i$ are multihomogeneous (and so homogeneous) in the variables $\xxi$, and homogeneous in the $t_i$, then we can think $Z$ as the incidence subvariety in $\PP^n\times (\P1)^{n+1}$, and in $\PP^n\times \PP^{2(n+1)-1}$. We will return to this fact to reduce the problem of analysing the homogeneous resultant in the space $\PP^{2(n+1)-1}$.

Write $Z_\Omega$ for the open set defined by the points $z\in Z$, such that $\pi_1(z)\in \Omega:=\PP^n-W$, where $W$ is the base locus, as in (\ref{WX}). The closed subscheme $Z$ of $\PP^n\times (\P1)^{n+1}$ corresponds to the projective scheme $\MProj(\BB)$ and $Z_\Omega$ is the open subset of $Z$ that is isomorphic to the complement of the base locus on $\PP^n$. For $p$ in the base locus of $\phi$, e.g. $p\in V(I^{(i)})$, there is a commutative diagram of schemes as follows:
\[
 \xymatrix@1{& Z_\Omega\ \ar@{^{(}->}[r]\ar@{-}[dd]& Z\ \ar@{^{(}->}[rr]\ar[dd]^{\pi_1}\ar[rrdd]^{\pi_2}& & \PP^n\times (\P1)^{n+1}\\ \pi_1^{-1}(p)\ar@{-}[dd]\ar@{^{(}->}[urr] & & & &\\ 
&\Omega\ \ar@{^{(}->}[r]&\PP^n\ar@{-->}[rr]^{\phi} && (\P1)^{n+1}.\\
\{p\} \ \ar@{^{(}->}[r] &V(I^{(i)})\ar@{^{(}->}[ur] & & & }
\]

Observe that the closed subscheme $Z$ is the Zariski closure of $Z_\Omega$, and that the closure of the second projection of $Z_\Omega$ coincides with the scheme theoretic image of $\phi$, this is $\pi_2(\overline{Z_\Omega})=\HH$. Assume that $\res_{A,d_0,\hdots,d_n}(L_0,\dots,L_n)\neq 0$, hence, this equation defines a divisor $[\pi_2(Z)]$ in $(\P1)^{n+1}$.

Due to the multihomogenity of the resultant, the polynomial $\res:=\res_{A,d_0,\hdots,d_n}(L_0,\dots,L_n)\in k[\X]$ is multihomogeneous, hence it defines a closed subscheme in $(\P1)^{n+1}$. By Theorem \ref{teo-elim}, this element is the image via the specialization map $ \epsilon: \ZZ[\textnormal{ coef}(\ffi):\  i=0,\hdots,n]\to k,$ of the irreducible equation $\res_{\ZZ,d_0,\hdots,d_n}(L_0,\hdots,L_n)\in \ZZ[\textnormal{ coef}(\ffi),\ \forall i][\X]$.

We claim that our hypotheses are the necessary ones to avoid extra factors. For this, 
set $\alpha:=i_0,\hdots,i_k$, and write 
\begin{equation}\label{Xalpha}
 X_{\alpha}:=\Proj(A/\sum_{j=0}^kI^{(i_j)}), \ X_{j}:=X_{\{j\}}\textnormal{ for just one }j \textnormal{, and }U_{\alpha}:=X_{\alpha}-\bigcup_{j\notin \alpha}X_j.
\end{equation}
Observe that $X_\alpha$ stands for the subset of $W$, containing $X$ (cf. \eqref{WX}), where the equations $L_{i_0},\hdots,L_{i_k}$ vanish identically. If $U_{\alpha}$ is non-empty, consider $p\in U_{\alpha}$, then $\dim(\pi_1^{-1}(p))=|\alpha|=k+1$. As the fibre is equidimensional, write $\EEE_{\alpha}:=\pi_1^{-1}(U_{\alpha})\subset \PP^n\times (\P1)^{n+1}$ for the the fibre over $U_{\alpha}$, which defines a (multi)projective bundle of rank $|\alpha|$. Hence, we have
\[
 \codim(\EEE_{\alpha})=n+1-(k+1)+(\codim_{\PP^n}(U_{\alpha})).
\]
The condition $\codim_A(I^{(i_0)}+\cdots+I^{(i_k)})> k+1$, for all $\{i_0,\hdots, i_k\}\subset \{1,\hdots,n\}$ implies that $\codim_{\PP^ n}(U_{\alpha})>k+1$. Hence 
\[
\codim(\EEE_{\alpha})>n+1=\codim(Z_\Omega). 
\]
Observe that $Z-Z_U=Z_\Omega$, where $Z_U:=\coprod_{\alpha}\EEE_{\alpha}$, and that $\codim(Z_U)>n+1=\codim(Z_\Omega)=\codim(\overline{Z_\Omega})$.

As $\Spec (B)$ is a complete intersection, in $\AA^{2(n+1)}$ it is unmixed and purely of codimension $n+1$. Thus, $Z\neq \emptyset$ is also purely of codimension $n+1$. This and the fact that $\codim (Z_U)>n+1$ implies that $Z=\overline{Z_\Omega}$.

The graph $Z_\Omega$ is irreducible, hence $Z$ is irreducible as well, and its projection (the closure of the image of $\phi$) is of codimension $1$, and one has for $\nu>\eta$:
\[
[\det((\k.)_\nu)]=\div_{k[\X]}(H_0(\k.)_\nu)=\div_{k[\X]}(\BB_\nu)=\sum_{\pp \textnormal{ prime, }\codim_{k[\X]}(\pp)=1}\length_{k[\X]_\pp} ((\BB_\nu)_\pp)[\pp].
\]
Also as $[\det((\k.)_\nu)]=\div_{k[\X]}(\res)=e.[\qq]$ for some integer $e$ and $\qq:=(H)\subset k[\X]$, we have that $\sum_{\pp \textnormal{ prime, }\codim(\pp)=1}\length_{k[\X]_\pp} ((\BB_\nu)_\pp)[\pp]=e.[\qq]$ hence $[\det((\k.)_\nu)]=\length_{k[\X]_\qq} ((\BB_\nu)_\qq)[\qq]$. Finally for $\nu\gg~0$
\[
 \length_{k[\X]_\qq} ((\BB_\nu)_\qq)=\dim_{\kk(\qq)}{(\BB_\nu\otimes_{k[\X]_\qq} \kk(\qq))}=\deg(\phi),
\]
where $\kk(\qq):=k[\X]_\qq/\qq k[\X]_\qq$. This shows that $e=\deg(\phi)$ and completes the proof.
\end{proof}

\begin{rem}\label{remHdivRes}
We showed that the scheme $\pi_2(Z)$ is defined by the polynomial $\res$, while the closed image of $\phi$ coincides with $\pi_2(\overline{Z_\Omega})$, hence the polynomial $H$ divides $\res$. Moreover, from the proof above we conclude that $H^{\deg(\phi)}$ also divides $\res$. And if $[\overline{\EEE_{\alpha}}]$ is an algebraic cycle of $\PP^n\times (\P1)^{n+1}$ of codimension $n+1$, then $[\overline{\pi_2(\EEE_{\alpha})}]$ is not a divisor in $(\P1)^{n+1}$, and consequently $\res$ has no other factor than $H^{\deg(\phi)}$.
\end{rem}

\medskip

\begin{rem}
 With the hypotheses of Theorem \ref{teoRes}(ii)(b), denoting by $\deg_i$ the degree on the variables $x_i,y_i$ and by $\deg_{tot}$ the total one, we have:
\begin{enumerate}
 \item $\deg_i(H)\deg(\phi)=\prod_{j\neq i}d_j$;
 \item $\deg_{tot}(H)\deg(\phi)=\sum_i\prod_{j\neq i}d_j$.
\end{enumerate}
\end{rem}

\medskip

In the rest of this section we will focus on the study of the difference between the closed schemes $\overline{Z_\Omega}$ and $Z$, and its projection $\pi_2(Z-\overline{Z_\Omega})$ on $(\P1)^{n+1}$. We will show how this can give extra factors on the resultant.

\subsection{Analysis of the extra factors}

In this section our aim is to analyse in detail the nature of the extra factors that appear when the codimension condition of Theorem \ref{teoRes}(ii)(b) does not hold. We will see that these polynomials appear due to the existence of big enough dimensional fibres over some of the components of the base locus. In a few words, it happens that if the sum of the dimension of a component in the base locus plus the dimension of its fibre is $n$, there is a ``big dimensional'' subscheme whose projection on $(\P1)^{n+1}$ defines a hypersurface which gives rise to a factor in the computed resultant.

In order to understand this, we will first analyse some simple cases, namely, where this phenomenon occurs over a finite set of points of the base locus; and later, we will deduce the general implicitization result.

\medskip

\begin{exmp}
 Assume we are given a rational map $\phi:\p2\dashrightarrow \P1 \times \P1 \times \P1$, where $\phi$ maps $(u:v:w)$ to $(f_0:g_0)\times(f_1:g_1) \times(f_2:g_2)$, of degrees $d, d'$ and $d''$ respectively. 

We may suppose that each of the pairs of polynomials $\{f_0,g_0\}$, $\{f_1,g_1\}$ and $\{f_2,g_2\}$ have no common factors. Then, the condition $\codim_A(I^{(i)})\geq 2$ is automatically satisfied. Assume also that $X=\emptyset$, this is, there are no common roots to all $6$ polynomials.

We will show here that, if we don't ask for the ``correct" codimension conditions, we could be implicitizing some extra geometric objects. For instance, suppose that we take a simple point $p\in V(I^{(0)}+I^{(1)})\neq \emptyset$. 
Consequently $L_0(u,v,w,\X)=L_1(u,v,w,\X)=0$ for all choices of $\X$. Nevertheless, $L_2(u,v,w,\X)=g_2(u,v,w)x_2-f_2(u,v,w)y_2=0$ imposes the nontrivial condition $g_2(p)x_2-f_2(p)y_2=0$ on $(\zz)$, hence there is one point $q=(f_2(p):g_2(p))\in \P1$ which is the solution of this equation. We get $\pi_1^{-1}(p)=\{p\}\times\P1\times \P1\times\{q\}$. As we do not want the reader to focus on the precise computation of this point $q$, we will usually write $\{*\}$ for the point $\{q\}$ obtained as the solution of the only nontrivial equation.

Suppose also that, for simplicity, $V(I^{(0)}+I^{(1)})=\{p\}$, $V(I^{(0)}+I^{(2)})=\emptyset$, and $V(I^{(1)}+I^{(2)})=\emptyset$. This says that if we compute $\pi_2(Z)$, then we get
\[
 \begin{array}{ll}
	\pi_2(Z)\ & = \pi_2(\pi_1^{-1}(\overline{\Omega}\cup W))=\pi_2(\overline{\pi_1^{-1}(\Omega))}\cup \pi_2(\pi_1^{-1}(W))=\\
	& = \pi_2(\overline{Z_\Omega})\cup (\pi_2( \{p\}\times \P1\times \P1\times\{*\})=\\
	& = \overline{\im(\phi)}\cup(\P1\times \P1\times\{*\}),
 \end{array}
\]
where $W =\Proj(A/\prod_{i}I^{(i)})$ the base locus of $\phi$ as in \eqref{WX}, and $\Omega=\PP^n-W$ its domain.
 
Hence, the Macaulay Resultant contains some extra factor. Let us observe that if there is only one extra hyperplane appearing (over a point $p$ with multiplicity one), which corresponds to $\pi_2(\pi_1^{-1}(p))$, then $\pi_1^{-1}(p)$ is a closed subscheme of $Z$,  defined by the equation $L_3(p)=0$. Then, we will show that 
\[
 \res_{(u,v,w)}(L_0,L_1,L_2)=\HH^{\deg(\phi)}\cdot L_3(p). 
\]
\end{exmp}

We will now generalize Theorem \ref{teoRes} in the spirit of the example above. For each $i\in \{0,\hdots,n\}$ take $X_{\hat{i}}:=\Proj(A/\sum_{j\neq i} I^{(j)})$. 

\medskip

\begin{prop}\label{casoFibraNdim}
Let $\phi: \PP^n\dashrightarrow (\P1)^{n+1}$ be a rational map that satisfies conditions $\ref{teo-kos}-\ref{teo-cod}$ of Theorem \ref{teoRes}. Assume further that for all $\{i_0,\hdots, i_k\}\subset \{0,\hdots,n\}$, with $k<n-1$, $\codim_A(I^{(i_0)}+\cdots+I^{(i_k)})> k+1$.
 Then, there exist non-negative integers $\mu_p$ such that: 
\[
\res_A(L_0,\dots,L_n)=H^{\deg(\phi)}\cdot \prod_{i=0}^n\prod_{p\in X_{\hat{i}}}L_i(p)^{\mu_{p}}.
\]
\end{prop}

\begin{proof}
 Denote by $\Gamma:=\overline{Z_\Omega}$ the closure of the graph of $\phi$, $Z$ as before. From Remark \ref{remHdivRes}, we can write 
\[
 G:=\frac{\res_A(L_0,\dots,L_n)}{H^{\deg(\phi)}},
\]
 the extra factor. It is clear that $G$ defines a divisor in $(\P1)^{n+1}$ with support on $\pi_2(Z-\Gamma)$. From the proof of Theorem \ref{teoRes}, we have that $Z$ and $\Gamma$ coincide outside $W\times (\P1)^{n+1}$. As $Z$ is defined by linear equations in the second group of variables, then $Z-\Gamma$ is supported on a union of linear spaces over the points of $W$, and so, its closure is supported on the union of the linear spaces $(\pi_1)^{-1}(p)\cong \{p\}\times ((\P1)^{n}\times \{*\})$, where $\{*\}$ is the point $(x:y)\in \P1$ such that $L_i(p,x,y)=0$ for suitable $i$. It follows that $\pi_2((\pi_1)^{-1}(p))\subset V(L_i)\subset (\P1)^{n+1}$, and consequently 
\[
 G=\prod_{p\in X}L_i(p)^{\mu_{p}},
\]
for some non-negative integers $\mu_p$.
\end{proof} 

As we said at the begining of this section, we are interested in understanding the extra factors. We have shown above that these factors depend directly on the dimension of each component of the base locus and on the dimension of its fibre; and that if their sum equals $n$ this generates an extra hypersurface, hence an extra factor. It is natural to ask, why this sum does not exceed $n$. The next lemma gives an answer to this question.

\medskip

\begin{lem}\label{LemaCodim}
 Let $\phi: \PP^n \dashrightarrow (\P1)^{n+1}$, be a finite rational map as above. Under the hypothesis of Theorem \ref{avramov}, $\codim_A(\sum_{j=0}^k I^{(i_j)})\geq k+1$ for all $k=0,\hdots,n$ and $\{i_0,\hdots,i_k\}\subset \{0,\hdots,n\}$.
\end{lem}

\begin{proof}
To show this we will use Theorem \ref{avramov}. Let us denote by $I$ the ideal $I^{(i_0)}+\cdots+I^{(i_n)}$, and for a fixed $k$ write $I=I'+I''$, where $I'=\sum_{j=0}^k I^{(i_j)}$ and $I''=\sum_{l=k+1}^n I^{(i_l)}$ the sum over the complementary indices of $\{i_0,\hdots,i_k\}$. As $(L_0,\hdots,L_n)$ is a complete intersection in $R$, also is $(L_{I_0},\hdots,L_{i_k})$ in $A[x_{i_0},y_{i_0},\hdots,x_{i_k},y_{i_k}]$. Applying Theorem \ref{avramov} to the ideal $(L_{I_0},\hdots,L_{i_k})$, for $r=1$ we have that $\codim_A(I^{(i_0)}+\cdots+I^{(i_k)})\geq k+1$.

Observe that as $I'$ is generated by a subset of the set of generator of $I$ then $I'$ is also a complete intersection in $R$. Now, as it is generated by elements only depending on the variables $x_{i_j},y_{i_j}$ for $j=0,\hdots,k$, we have that it is also a complete intersection in $A[x_{i_0},y_{i_0},\hdots,x_{i_k},y_{i_k}]$.
\end{proof}

\begin{rem}
 It is easy to see that the converse of Lemma \ref{LemaCodim} does not always hold (in fact almost never). For example let us take for $n=1$, $f_0=g_0=f_1=g_1\neq 0$.
\end{rem}

\medskip

We will now introduce some notation we will use to state and prove Theorem \ref{teoResGral}. The idea behind these definitions is to split the base locus into irreducible subvarieties with the desired dimension, in order to obtain a clear factorization of the extra component that appears.

\medskip

\begin{defn}
For each $\alpha:=(i_0,\hdots,i_k)$ set $I^{\alpha}:=\sum_{j=0}^k I^{(i_j)}$, and $X_{\alpha}:=\Proj(A/I^{\alpha})$ as defined in \eqref{Xalpha}. Denote by $\Theta$ the set of all $\alpha\subset \{0,\hdots,n\}$ such that $\codim(I^{\alpha})=|\alpha|$.

For each $\alpha:=(i_0,\hdots,i_k) \in \Theta$, let $I^\alpha=(\cap_{\qq_i\in \Lambda_\alpha} \qq_i)\cap \qq'$, where $\Lambda_\alpha$ is the set of primary ideals of codimension $|\alpha|$ containing $I$, and $\codim_A(\qq')>|\alpha|$. Write
\begin{equation}
 X_{\alpha,i}:=\Proj(A/\qq_i), \qquad \textnormal{ with } \qq_i\in \Lambda_\alpha,
\end{equation}
and let $X_{\alpha,i}^{red}$ be the associated reduced variety. Denote by $\pi_\alpha$ the projection onto the coordinates $(x_i:y_i)$, for $i\notin \alpha$. Namely, let $\{i_{k+1},\hdots,i_n\}:=\{1,\hdots,n\}-\alpha$, then $\pi_\alpha$ is given by 
\begin{equation}
\begin{array}{rcl} 
 \pi_\alpha: (\P1)^{n+1}&\to &(\P1)^{n+1-|\alpha|} \\
(x_0:y_0)\times \hdots \times(x_n:y_n)&\mapsto &(x_{i_{k+1}}:y_{i_{k+1}})\times \hdots \times(x_{i_{n}}:y_{i_{n}}).
\end{array}
\end{equation}
Set $\phi_{\alpha}:\PP^n\dto \PP_\alpha:= \pi_{\alpha}((\P1)^{n+1})$ defined as $\phi_{\alpha}:= \pi_{\alpha}\circ \phi$. Hence, $\phi_{\alpha}$ is explicitly defined as
 $\phi_{\alpha}(\t)=(f_{i_{k+1}}(\t):g_{i_{k+1}}(\t))\times \hdots \times(f_{i_{n}}(\t):g_{i_{n}}(\t))$ as is shown in the diagram below
\[
 \xymatrix@1{
 \PP^n\ar@{-->}[r]^{\phi}\ar@{-->}[rd]^{\phi_{\alpha}}& (\P1)^{n+1}\ar^{\pi_{\alpha}}[d]&\\
 &  **[r]\PP_\alpha:= \pi_{\alpha}((\P1)^{n+1}) &\ \ \ \cong (\P1)^{n+1-|\alpha|}. }
\]
 Denote by $W_\alpha$ the base locus of $\phi_{\alpha}$ and $X:=\Proj(A/\sum_{i}I^{(i)})$. Since $W:=\Proj(A/\prod_{i}I^{(i)})$ is the base locus of $\phi$ as defined in (\ref{WX}), we have the inclusions $X\subset W_{\alpha}\subset W$. 

Denote by $\UU_{\alpha}:=\PP^n-W_{\alpha}$, the open set where $\phi_{\alpha}$ is well defined. Write $\Omega_{\alpha}:= X_{\alpha}\cap \UU_{\alpha}$ and $\Omega_{\alpha,i}:= X_{\alpha,i}\cap \UU_{\alpha}$. The open set $\Omega_{\alpha,i}$ stands for the points of $\PP^n$ which give rise to an extra factor. This factor coincides with the restricted image of $\phi$ over $\Omega_{\alpha,i}$ to a certain power. 

If $\alpha$ is empty, we set $\pi_\alpha=Id_{(\P1)^{n+1}}$, $\phi_\alpha=\phi$, $W_\alpha=W$ and $\UU_{\alpha}=\Omega_{\alpha}=\Omega$.
\end{defn}

\medskip

Next, we show that this meticulous decomposition of the base locus has the properties we are looking for. Namely, the fibre over each of the open sets $\Omega_{\alpha,i}$ has dimension $|\alpha|$, in such a way that the sum of the dimensions of the base plus the dimension of the fiber is $n$.

\medskip

\begin{lem}\label{LemBundle}
 Let $\phi$ be as in Theorem \ref{teoRes}. For each $\alpha \in \Theta$ and each $\qq_i\in \Lambda_\alpha$, the following statements are satisfied:
\begin{enumerate}
 \item $\Omega_{\alpha,i}$ is non-empty.
 \item for all $p\in \Omega_{\alpha,i}$, $\dim(\pi_1^{-1}(p))=|\alpha|$.
 \item \label{bundle} the restriction $ \phi_{\alpha,i}$ of $\phi$ to $\Omega_{\alpha,i}$, defines a rational map 
\begin{equation}\label{phiai}
 \phi_{\alpha,i}:X_{\alpha,i}\dto \PP_\alpha\cong (\P1)^{n+1-|\alpha|}.
\end{equation}
\item $Z_{\alpha,i}:=\pi_1^{-1}(\Omega_{\alpha,i})$ is a (multi)projective bundle $\EEE_{\alpha,i}$ of rank $|\alpha|$ over $\Omega_{\alpha,i}$.
\end{enumerate}
\end{lem}

\begin{proof}
Fix $X_{\alpha,i}\subset X_{\alpha}$ and write $\alpha:=i_0,\hdots,i_k$. Note that $\Omega_{\alpha,i}:=X_{\alpha,i}-\bigcup_{j\notin \alpha}X_{\{j\}}$ is an open subset of $X_{\alpha,i}$. If $\Omega_{\alpha,i}=\emptyset$, then $X_{\alpha,i}\subset \bigcup_{j\notin \alpha}X_{\{j\}}$, and since it is irreducible, there exists $j$ such that $X_{\alpha,i}\subset X_{\{j\}}$, hence $X_{\alpha,i}\subset X_{\{j\}}\cap X_\alpha= X_{\alpha\cup \{j\}}$. Setting $\alpha':= \alpha\cup \{j\}$, it follows that $\dim(X_{\alpha'})\geq \dim(X_{\alpha,i})=n-|\alpha|>n-|\alpha'|$, which contradicts the hypothesis.

Let $p\in \Omega_{\alpha,i}$, $\pi_1^{-1}(p) = \{p\}\times\{q_{i_{k+1}}\}\times\hdots\{q_{i_{n}}\}\times(\P1)^{|\alpha|}$, where the point $q_{i_{j}}\in \P1$ is the only solution to the nontrivial equation $L_{i_{j}}(p,x_{i_{j}},y_{i_{j}})=y_{i_{j}}f_{i_{j}}(p)-x_{i_{j}}g_{i_{j}}(p)=0$. Then we deduce that $\dim(\pi_1^{-1}(p))=|\alpha|$, and that $\phi_{\alpha,i}:\Omega_{\alpha,i}\to \prod_{j=k+1}^{n}\P1_{(i_j)}\cong (\P1)^{n+1-|\alpha|}$ given by $p\in \Omega_{\alpha,i}\mapsto \{q_{i_{k+1}}\}\times\hdots\times\{q_{i_{n}}\}\in \prod_{j=k+1}^{n+1}\P1_{(i_j)}$, is well defined.

The last statement follows immediately from the previous ones. 
\end{proof}

Note that $X_{\alpha,i}$ has dimension $n-|\alpha|$, by the preceding lemma $\dim(\pi_1^{-1}(\Omega_{\alpha,i}))=n$. Since $\res_A(L_0,\hdots,L_n)$ describes the codimension one part of $\pi_2(Z)$, if $\dim(\pi_2(\pi_1^{-1}(\Omega_{\alpha,i})))=n$, then $\res_A(L_0,\hdots,L_n)$ is not irreducible. Denote by $\Delta_\alpha\subset \Lambda_\alpha$ the subset of primary ideals $\qq_i$ satisfiying  $\dim(\pi_2(\pi_1^{-1}(X_{\alpha,i})))=n$. Observe that if $|\alpha|=n$ we are in the case of Proposition \ref{casoFibraNdim}. The following diagram illustrates this situation:
\[
\xymatrix@1{& & Z\ \ar@{^{(}->}[rr]\ar[dd]^{\pi_1}\ar[rrdd]^{\pi_2}& & \PP^n\times  (\P1)^{n+1}\\
 \pi_1^{-1}(p)\  \ar@{-}[d]\ar@{^{(}->}[r] & \ \EEE_{\alpha,i} \ \ \ar@{-}[d] \ar[ddrr]\ar@{^{(}->}[ur]& & &\\
 \{p\} \ \ar@{^{(}->}[r] & \ \ \Omega_{\alpha,i} \ar@{^{(}->}[d]\ar[rrd]^{\phi_{\alpha,i}} \ \ar@{^{(}->}[r]  &\PP^n\ar@{-->}[rr]^{\phi} & & (\P1)^{n+1}\ar^{\pi_{\alpha}}[d].\\
 & X_{\alpha,i}\ar@{^{(}->}[ur]\ \ar@{-->}[rr]^{\phi_{\alpha,i}} & & \pi_2(\EEE_{\alpha,i}) \  \ar@{^{(}->}[r]& (\P1)^{n+1-|\alpha|}, }
\]
where in this case $\pi_1^{-1}(p) = \{p\}\times\{q_{i_{k+1}}\}\times\hdots\times\{q_{i_{n}}\}\times(\P1)^{|\alpha|}$ is the $|\alpha|$-dimensional fibre of $\EEE_{\alpha,i}$ over $p$, for some $q_{i_j}\in \P1$. 

We finally state our general result.

\medskip

\begin{thm}\label{teoResGral} Let $\phi: \PP^n\dashrightarrow (\P1)^{n+1}$ be a finite rational map, satisfying the hypotheses of Theorem \ref{teoRes} and conditions $(i.a)-(i.c)$ in that theorem, hence $\RRes\neq 0$. Denote by $H$ the irreducible implicit equation of the closure of the image of $\phi$, and by $H_{\alpha,i}$ the irreducible implicit equation of the closure of the image of $\phi_{\alpha,i}$ defined in (\ref{phiai}) for each $\alpha \in \Theta$, and $i$ such that $\qq_i\in \Delta_\alpha$. 

Then, there exist positive integers $\mu_{\alpha,i}$ such that:
\[
 \res_A(L_0,\hdots,L_n)=H^{\deg(\phi)}\cdot \prod_{\alpha\in \Theta}\prod_{i: \qq_i\in \Delta_\alpha} (H_{\alpha,i})^{\mu_{\alpha,i}\deg(\phi_{\alpha,i})}.
\]
\end{thm}

\begin{proof}
 Denote by $\Gamma:=\overline{Z_\Omega}$, the graph of $\phi$, $Z$ the incidence scheme, and $G=\frac{\res_A(L_0,\dots,L_n)}{H^{\deg(\phi)}}$ the extra factor.

As in Proposition \ref{casoFibraNdim}, $G$ defines a divisor in $(\P1)^{n+1}$ with support on $\pi_2(Z-\Gamma)$, and $Z$ and $\Gamma$ coincide outside $W\times (\P1)^{n+1}$. 

By part \eqref{bundle} of Lemma \ref{LemBundle}, for each $\alpha$ and each $\qq_i\in \Delta_\alpha\subset\Lambda_\alpha$, $\phi_{\alpha,i}$ defines a (multi)projective bundle $\EEE_{\alpha,i}$ of rank $|\alpha|$ over $\Omega_{\alpha,i}$. 
By definition of $\Delta_\alpha$, $\overline{\pi_2(\EEE_{\alpha,i})}$ is a closed subscheme of $(\P1)^{n+1}$ of codimension $1$. Denoting by $[\overline{\EEE_{\alpha,i}}]={\mu_{\alpha,i}}.[\overline{\EEE_{\alpha,i}^{\ red}}]$ the class of $\overline{\EEE_{\alpha,i}}$ as an algebraic cycle of codimension $n+1$ in $\PP^n\times(\P1)^{n+1}$, we have $(\pi_2)_*[\overline{\EEE_{\alpha,i}}]= {\mu_{\alpha,i}}.(\pi_2)_*[\overline{\EEE_{\alpha,i}^{\ red}}]= {\mu_{\alpha,i}}.\deg(\phi_{\alpha,i}).[\pp_{\alpha,i}]$, where $\pp_{\alpha,i}:=(H_{\alpha,i})$. 

As in Theorem \ref{teoRes}, one has for $\nu>\eta$:
\[
[\det((\k.)_\nu)]=\div_{k[\X]}(H_0(\k.)_\nu)=\div_{k[\X]}(\BB_\nu)=\sum_{\pp \textnormal{ prime, }\codim_{k[\X]}(\pp)=1}\length_{k[\X]_\pp} ((\BB_\nu)_\pp)[\pp].
\]
We obtain that
\[
 [\det((\k.)_\nu)]= \sum_{\alpha\in \Theta} \sum_{\pp_{\alpha,i}} \length_{k[\X]_{\pp_{\alpha,i}}} ((\BB_\nu)_{\pp_{\alpha,i}})[\pp_{\alpha,i}]+\length_{k[\X]_{(H)}} ((\BB_\nu)_{(H)})[(H)].
\]

In the formula above, for each $\pp_{\alpha,i}$ we have
\[ 
\length_{k[\X]_{\pp_{\alpha,i}}} ((\BB_\nu)_{\pp_{\alpha,i}})=\dim_{\kk(\pp_{\alpha,i})}{(\BB_\nu\otimes_{k[\X]_{\pp_{\alpha,i}}} \kk(\pp_{\alpha,i}))}={\mu_{\alpha,i}}.\deg(\phi_{\alpha,i}),
\]
where $\kk(\pp_{\alpha,i}):=k[\X]_{\pp_{\alpha,i}}/\pp_{\alpha,i} k[\X]_{\pp_{\alpha,i}}$. Consequently we get that for each $\alpha\in \Theta$, there is a factor of $G$, denoted by $H_{\alpha,i}$, that corresponds to the irreducible implicit equation of the scheme theoretic image of $\phi_{\alpha,i}$, raised to a certain power $\mu_{\alpha,i}.\deg(\phi_{\alpha,i})$.
\end{proof}

\begin{rem}
It is important to remark that the set-theoretic approach doesn't tell us anything about the scheme structure of the fibre $(\pi_1)^{-1}(W)$. Namely the bijection 
\[
\bigcup_\alpha (X_\alpha\times ((\P1)^{|\alpha|}\times \{*\}\times \hdots \times \{*\}) ) \stackrel{\sim}{\lto} (\pi_1)^{-1}(W)
\]
 is not necessarily a scheme isomorphism; for instance, $\EEE_{\alpha,i}$ need not be well-defined over all $X_{\alpha,i}$, because the multiplicities of the components of $X_{\alpha,i}$ of $\Proj(A/I^\alpha)$ are not necessarily preserved by $\pi_1^{-1}$. The fiber bundle $\EEE_{\alpha,i}$ is defined over the relative open set $\Omega_{\alpha,i}$ of $X_{\alpha,i}\subset \Proj(A/I^\alpha)$, but the dimension of the fibre can increase on $X_{\alpha,i}-\Omega_{\alpha,i}$.
\end{rem}

\medskip

\begin{rem}\label{remBoundMu}Let us discuss briefly how the exponents $\mu_{\alpha,i}$ behave in the particular case where $X_{\alpha,i}$ is a complete intersection subscheme of $\PP^n$. Take $\alpha\in \Theta$, defined as above, write $\alpha=\{i_0,\hdots,i_k\}$, and fix $\pp_i\in \Lambda_\alpha$. As $\codim(I^{\alpha,i})=|\alpha|$, we have $\dim_{\PP^n}(X_{\alpha,i})=n-|\alpha|=n-(k+1)=:m$, and assume $X_{\alpha,i}$ is irreducible.

Take $G_0,\hdots,G_{k}$ irreducible generators of $I^{\alpha,i}$ (recall $X_{\alpha,i}$ is a complete intersection), and $G_j:=L_{i_j}$ for $j\in \{k+1,\hdots,n\}$. 

As $G_0,\hdots,G_{k}$ vanish over $X_{\alpha,i}$, the element $\res(G_0,\hdots,G_n)\in k[x_{i_{k+1}},y_{i{k+1}},\hdots,x_{i_n},y_{i_n}]$ describes exactly the irreducible implicit equation, $H_{\alpha,i}$, of the scheme theoretic image, $\HH_{\alpha,i}\subset \P1_{i_{k+1}}\times \hdots \times \P1_{i_n}$, of the restricted and correstricted map $\phi_{\alpha,i}$.

Assume that $L_j$ lies in $G:=(G_0,\hdots,G_n)^{\mu_j}$, and that $\mu_j$ is maximum with this property, for all $j$. Then, from the ``Lemme de divisibilit\'e g\'en\'eral" by J.-P. Jouanolou (see \cite[Prop. 6.2.1]{Jou3}), we have $\res(G_0,\hdots,G_n)=H_{\alpha,i} ^{\prod_j\mu_j}$. As this polynomial divides $H_{\alpha,i} ^{\mu_{\alpha,i}}$, $\prod_j\mu_j$ gives a lower bound for $\mu_{\alpha,i}$.

Remark also that $L_{i_j}$ always lies in the ideal $G:=(G_0,\hdots,G_n)^{\mu_j}$ for $\mu_{i_j}=1$ by definition, when $j>k$. 
\end{rem}


\section{Examples}

In this section we present several examples where we illustrate the theory developed in the previous sections. These computations were done in \textsl{Maple 11}, by means of the routines implemented in the \textsl{Multires} developed by Galaad Team at INRIA, cf. \cite{GalMpl}.

\medskip

\begin{exmp}
In this example we show what happens when there is a point whose fibre has dimension $2$.

{\small \begin{verbatim}

> read"multires.mpl": with(linalg):
> f0 := u: g0 := v: f1 := u^2: g1 := v^2: f2 := v^2: g2 := w^2:
                                2       2              2       2
L0 := x0 v - y0 u,    L1 := x1 v  - y1 u,    L2 := x2 w  - y2 v

> M012w:=det(mresultant([L0,L1,L2],[u,v])); 

                               2  4          2        2 2
                   M012w := -x2  w  y1 (y1 x0  - x1 y0 )  y2

\end{verbatim}
}
As \textsl{mresultant} gives a multiple of the desired Macaulay resultant, computing the greatest common divisor over all permutations of $L_0, L_1, L_2$ we get that $Mw$ should be:

{\small \begin{verbatim}

                            2  4       2        2 2
                    Mw := x2  w  (y1 x0  - x1 y0 )

\end{verbatim}
}
It remains to observe that $x_2^2$ correspond to the equation $L_2(p)=0$, where $p=(0:0:1)$ is the point with $2$-dimensional fibre, and use the fact that $p$ has multiplicity $2$.
\end{exmp}

\medskip

\begin{exmp}
In this example we study the case of a two point ($p_1$ and $p_2$) base locus, where the fibre above them has dimension $2$. Take $p_1=(1:0:0)$ and $p_2=(0:0:1)$. The computation below shows that in this case the methed yields a power of the irreducible implicit equation with some extra factors. Those factors are powers of $y_1$ and $x_2$. 

Why $y_1$ and $x_2$? If we look at the linear equations $L_0$, $L_1$, $L_2$ evaluated in the two points, $p_1$, $p_2$, we see that $L_0(p_1)=L_2(p_1)=0$ and $L_1(p_1)=y_1$. If we do the same with $p_2$ we get $L_0(p_2)=L_1(p_2)=0$ and $L_2(p_2)=x_2$.

{\small \begin{verbatim}

> f0 := u*w: g0 := v^2: f1 := u^2: g1 := v^2: f2 := v^2: g2 := w^2: 
> L0:=x0*g0-y0*f0: L1:=x1*g1-y1*f1: L2:=x2*g2-y2*f2: 

\end{verbatim}
}
As we did before, we compute the greatest common divisor over all the permutation of $L_0, L_1, L_2$. The ouput $Mw$, $Mv$ and $Mu$ is:
{\small \begin{verbatim}

                           2  8   3           2        2    2
                    Mw:= x2  w  y1  (-y2 x1 y0  + x2 x0  y1)

                          8   3   4           2        2    2
                    Mv:= v  y1  x2  (-y2 x1 y0  + x2 x0  y1)

                           2  8   4           2        2    2
                    Mu:= y1  u  x2  (-y2 x1 y0  + x2 x0  y1)

\end{verbatim}
}
The resultant can be obtained as the gcd of the three equations above.
\end{exmp}

\medskip

\begin{rem}
In the spirit of Remark \ref{remBoundMu} take $\alpha=\{0,1\}$, that is, $X_\alpha=V(f_0,g_0,f_1,g_1)=V(u,v)$, and consider $G_0=u$, $G_1=v$, and $G_2=L_2=y_2v^2-x_2w^2$. By the multiplicativity of the resultant (see for instance \cite[Sec. 5.7]{Jou3}) we have $\res(G_0,G_1,G_2)=\res(u,v,-x_2w^2)=-x_2$. Now, as $L_0\in G:=(G_0,G_1,G_2)$, but not in $G^2$, and $L_1\in G^2$, but not in $G^3$, we have $\mu_0=1$ and $\mu_1=2$. Hence, $\res(G_0,G_1,G_2)^2=x_2^2$ divides $\res(L_0,L_1,L_2)$, as the above computation showed. Moreover, in this case, we see that $\res(G_0,G_1,G_2)^2=x_2^2$ coincides exactly with the extra factor.
\end{rem}

\medskip

\begin{exmp}
In this example we study the situation where the fibre along a $1$-dimensional closed subscheme of $\PP^2$ is $\P1$, and is $\PP^2$ above the points $(1:0:0)$ and $(0:0:1)$. 

{\small \begin{verbatim}

> f0 := u*v: g0 := u*w: f1 := u^2+v^2: g1 := v^2: f2 := v^2: g2 :=w^2: 
> L0:=x0*g0-y0*f0; L1:=x1*g1-y1*f1; L2:=x2*g2-y2*f2; 

                      8   2   3             2     2        2    2
               Mw:=  w  x2  y1  y2 (x1 - y1)  (-x0  y2 + y0  x2)

                      8   4   3          2     2        2    2
               Mv:=  v  x2  y1  (x1 - y1)  (-x0  y2 + y0  x2)

                      8   2     2        2    2
               Mu:=  u  y1  (-x0  y2 + y0  x2)

\end{verbatim}
}

This last equation corresponds to the situation on the open set $u\neq 0$. Is clear that the extra factor $y_1$ is appearing because of the $2$ dimensional fibre above $p=(1:0:0)$, where $L_0(p)=L_2(p)=0$, and $L_1(p)=u^2 y_1$. Similarly $x_2$ appears because of the $2$ dimensional fibre above the point $q=(0:0:1)$, where $L_0(q)=L_1(q)=0$, and $L_2(q)=w^2 x_2$, as was shown in the second example. The other factor that appears is $(x_1-y_1)^2$, and it is due to the existence of a $1$ dimensional closed subvariety $V$. In this case, we see that the dimension of the fibre over a point $p$ is $1$, for any $p$ in a suitable relative open subset of $V$. 

Precisely, we consider the fibre along the closed subvariety of $\PP^2$, $V(u)=\{(u:v:w)\ : \ u=0\}$ and we project from this closed subvariety of the incidence variety to the space $\P1\times \P1$ corresponding to the second two copies (because the first one would be $(0:0)$).
If we compute the implicit equation by the method before, for the map 
\[
 \begin{array}{rccc} \phi|_{(u=0)}:(u=0)\cong & \P1 & \dto & \P1\times \P1\ \\
 & (v:w) & \mapsto & (f_1:g_1)\times(f_2:g_2),
 \end{array}
\]
we get as output: $Mv:= -v^4\ (x1 - y1)^2$ and $Mw:= x^4\ (x1 - y1)^2 x2^2$. Observe that the $x_2^2$ appearing in the second equation is still the extra factor coming from the big fibre over the point $q$ in the closed set $V(u)$. We see here that the dimension of the fibre is not constant on $V(u)$, but on a relative open set where it defines a fibre bundle $\EEE$. 
\end{exmp}


\section{Applications to the computation of sparse discriminants}

The computation of sparse discriminants is equivalent to the implicitization problem for a parametric variety, to which we can apply the techniques developed in the previous sections. In the situation described in \cite{CD}, a rational map $f: \CC^n \dto \CC^n$ given by homogeneous rational functions of total degree zero is associated to an integer matrix $B$ of full rank. This is done in such a way that the corresponding implicit equation is a dehomogenization of a sparse discriminant of generic polynomials with exponents in
a Gale dual of $B$.

\medskip

Suppose for instance that we take the matrix $B$ below:
\[
 B=\left( \begin{array}{rrr}1&0&0\\-2&1&0\\1&-2&1\\0&1&-2\\0&0&1\end{array}\right).
\]
In this case, as the colums of $B$ generate the affine relations of the lattice points $\{0,1,2,3,4\}$. The closed image of the parametrization $f$ is a dehomogenization of the classical discriminant of a generic univariate polynomial of degree $4$. Explicitly, from the matrix we get the linear forms $l_1(u,v,w)=u, \ l_2(u,v,w)=-2u+v, \ l_3(u,v,w)=u-2v+w, \ l_4(u,v,w)=v-2w, \ l_5(u,v,w)=w$ (whose coefficients are read in the rows of $B$), and the polynomials $f_0=l_1\cdot l_3, \ g_0=l_2^2, \ f_1=l_2\cdot l_4, \ g_1=l_3^2, \ f_2=l_3\cdot l_5, \ g_2=l_4^2$ (the exponents of the linear
forms are read from the columns of $B$) . This construction gives rise to the following rational map: 
\[
 \begin{array}{rcl} f: \CC^3 & \dto & \CC^3 \\
 (u,v,w) & \mapsto & (\frac{u(u-2v+w)}{(-2u+v)^2},\frac{(-2u+v)(v-2w)}{(u-2v+w)^2},\frac{(u-2v+w)w}{(v-2w)^2}).
 \end{array}
\]

First, we see that we can get a map from $\PP^2_\CC$ because of the homogeneity of the polynomials. Also, taking common denominator, we can have a map to $\PP^3_\CC$, this is:
\[
 \begin{array}{rcl} f: \PP^2_\CC & \dto & \PP^3_\CC\\
 (u:v:w) & \mapsto & (f_0:f_1 :f_2 :f_3).
 \end{array}
\]
where $f_0= (-2u+v)^2(u-2v+w)^2(v-2w)^2$ is the common denominator, $f_1= u(u-2v+w)^3(v-2w)^2$, $f_2 =(-2u+v)^3(v-2w)^3 $ and $f_3=(u-2v+w)w(-2u+v)^2(u-2v+w)^2$.

The problem with this way of projectivizing is that, in general, we cannot implement the theory developed by L. Bus\'e, M. Chardin, and J-P. Jouanolou, because typically the base locus has awful properties, as a consequence of taking common denominator.

As a possible way out, we propose in this work to consider the morphism of projective schemes given by:
\[
 \begin{array}{rcl} \phi:\p2 & \dto &\P1 \times \P1 \times\P1\\
 (u:v:w) & \mapsto & (f_0:g_0)\times(f_1:g_1)\times(f_2:g_2).
 \end{array}
\]
where $f_0=u(u-2v+w)$, $g_0=(-2u+v)^2$, $f_1=(-2u+v)(v-2w)$, $g_1=(u-2v+w)^2$, $f_2=(u-2v+w)w$ $g_2=(v-2w)^2$.
For this particular example, we get that there are only two base points giving rise to an extra factor, namely $p=(1:2:3)$ and $q=(3:2:1)$. Is easy to see that those points give rise to two linear factors in the resultant.

\medskip

First, we observe that this situation is better, because we are not adding common zeroes. Moreover, if a point $(u:v:w)$ is a base point here, it also is in the two settings above.

Remember also that in the $n=2$ case, the condition required on the Koszul complex associated to this map for being acyclic is that the variety $X$, defined as the common zeroes of all the $6$ polynomials, be empty. In general, the conditions we should check are the ones imposed by the Avramov's theorem, as was shown in Theorem \ref{teoRes}. 

Note also that if we want to state this situation in the language of approximation complexes, we need only to replace $\k.$ by $\Z.$, because we can assume that $\{\ffi\}$ are regular sequences, due to the fact that $\gcd(\ffi)=1$.

\medskip

\begin{rem}
For a matrix like the $B$ above, it is clear that the closed subvariety $X$ is always empty, due to the fact that all maximal minors of $B$ are not zero, and the polynomials $g_i$'s involve independent conditions. Then, the only common solution to $l_2^2=l_3^2=l_4^2=0$ is $(u,v,w)=(0,0,0)$,  and so $X=\emptyset$ in $\p2$.
In this case, it is still better (from an algorithmic approach) to compute the discriminant of a generic polynomial of degree $4$ in a single variable and then dehomogenize, because, in our  setting, the number of variables is bigger than $1$. But when the
number of monomials of a sparse polynomial in many variables is not big, this Gale dual approach for the computation of
sparse discriminants provides a good alternative.
\end{rem}

We will give next an example where we show a more complicated case.

\medskip

\begin{exmp}
Let $C$ be the matrix given by 
\[
 C=\left( \begin{array}{rrr}1&-7&-6\\-1&4&3\\1&0&4\\0&1&-1\\-1&2&0\end{array}\right).
\]
As before, denoting by $b_i$ the $i$-th row of $C$, we get the linear forms 
$l_i(u,v,w)=\gen{b_i,(u,v,w)},$ associated to the row vectors $b_i$ of $B$,
where $\gen{,}$ stands for the inner product in $\CC^3$. Then we define the homogeneous polynomials $f_0=l_1\cdot l_3, \ g_0=l_2\cdot l_5, \ f_1=l_2^4\cdot l_4\cdot l_5^2, \ g_1=l_1^7, \ f_2=l_2^3\cdot l_3^4, \ f_2=l_1^6\cdot l_4$. And we obtain the following rational map:
\[
\begin{array}{rcl}\phi:\p2 & \dto & \P1 \times \P1 \times\P1 \\
 (u:v:w) & \mapsto & (f_0:g_0)\times(f_1:g_1)\times(f_2:g_2).
\end{array}
\]

It is easy to see that in this case the variety $X$ is not empty, for instance the point $p=(1:1:-1)$, defined by $l_1=l_2=0$ belongs to $X$.

As was shown by M. A. Cueto and A. Dickenstein in \cite[Lemma 3.1 and Thm. 3.4]{CD}, we can interpret the discriminant computed from the matrix $C$ in terms of the dehomogenized discriminant associated to any matrix of the form $C \cdot M$, where $M$ is a square invertible matrix with integer coefficients. That is, we are allowed to do operations on the columns of the matrix $C$, and still be able to compute the desired discriminant in terms of the matrix obtained from $C$. In \cite{CD} they give an explicit formula for this passage. 

In this particular case, we can multiply $C$ from the right by a determinant $1$ matrix $M$, obtaining
{\small \[
C\cdot M=\left( \begin{array}{rrr}1&-7&-6\\-1&4&3\\1&0&4\\0&1&-1\\-1&2&0\end{array}\right) \cdot \left( \begin{array}{rrr}1&12&-1\\0&6&-1\\0&5&1\end{array}\right)=\left( \begin{array}{rrr}1&0&0\\-1&-3&0\\1&-8&3\\0&11&-2\\-1&0&-1\end{array}\right). 
\]}

Similar to what we have done before, we can see that the closed subvariety $X$ associated to the rational map that we obtain from the matrix $C\cdot M$ is empty. Observe that $\#V(I_2)$ is finite due to the fact that $l_2=l_4=0$ or $l_3=l_4=0$ or $l_3=l_5=0$ should hold.  Moreover it is easy to verify that all maximal minors are nonzero, and this condition implies that any of the previous conditions define a codimension $2$ variety, this is, a finite one. A similar procedure works for seeing see that $\codim_A(I_3)\geq 2$. Finally the first part of Theorem \ref{teoRes} implies that the Koszul complex $\k.$ is acyclic and so we can compute the Macaulay resultant as its determinant. 

Moreover, this property over the minors implies that $\codim_A(I^{(i_0)})=2>k+1=1$ and that $\codim_A(I^{(i_0)}+I^{(i_1)})=3>k+1=2$. So, the second part of Theorem \ref{teoRes} tells us that the determinant of the Koszul complex $\k.$ in degree greater than $(2+8+3)-3=10$ determines exactly the implicit equation of the scheme theoretic image of $\phi$. Observe that, as was shown in \cite[Thm. 2.5]{CD}, for this map, we have that $\deg(\phi)=1$.
\end{exmp}

We remark that the process implemented for triangulating the matrix $C$ via $M$ is not algorithmic for the moment. 

\medskip

\subsection*{Acknowledgements\markboth{Acknowledgements}{Acknowledgements}}
 I would like to thank Laurent Bus\'e, and my two advisors: Marc Chardin and Alicia Dickenstein, for the very useful discussions, ideas and suggestions. Also to the Galaad group at INRIA, for their hospitality. Finally I would like to thank the reviewer for his huge dedication and many corrections to this article.

\medskip

\def\cprime{$'$}

\end{document}